\pdfoutput=1
\RequirePackage{ifpdf}
\ifpdf 
\documentclass[pdftex]{sigma}
\else
\documentclass{sigma}
\fi

\numberwithin{equation}{section}

\newtheorem{Theorem}{Theorem}[section]
\newtheorem{Corollary}[Theorem]{Corollary}
\newtheorem{Lemma}[Theorem]{Lemma}
 { \theoremstyle{definition}
\newtheorem{Definition}[Theorem]{Definition}
}

\begin{document}

\allowdisplaybreaks

\newcommand{\arXivNumber}{1706.05050}

\renewcommand{\PaperNumber}{050}

\FirstPageHeading

\ShortArticleName{Topology of Functions with Isolated Critical Points on the Boundary}

\ArticleName{Topology of Functions with Isolated Critical Points \\ on the Boundary of a 2-Dimensional Manifold}

\Author{Bohdana I.~HLADYSH and Aleksandr O.~PRISHLYAK}
\AuthorNameForHeading{B.I.~Hladysh and A.O.~Prishlyak}
\Address{Faculty of Mechanics and Mathematics, Taras Shevchenko National University of Kyiv,\\ 4-e Akademika Glushkova Ave., Kyiv, 03127, Ukraine}
\Email{\href{mailto:bohdanahladysh@gmail.com}{bohdanahladysh@gmail.com}, \href{mailto:prishlyak@yahoo.com}{prishlyak@yahoo.com}}

\ArticleDates{Received November 18, 2016, in f\/inal form June 16, 2017; Published online July 01, 2017}

\Abstract{This paper focuses on the problem of topological equivalence of functions with isolated critical points on the boundary of a compact surface $M$ which are also isolated cri\-tical points of their restrictions to the boundary. This class of functions we denote by~$\Omega(M)$. Firstly, we've obtained the topological classif\/ication of above-mentioned functions in some neighborhood of their critical points. Secondly, we've constructed a chord diagram from the neighborhood of a critical level. Also the minimum number of critical points of such functions is being considered. And f\/inally, the criterion of global topological equivalence of functions which belong to~$\Omega(M)$ and have three critical points has been developed.}

\Keywords{topological classif\/ication; isolated boundary critical point; optimal function; chord diagram}

\Classification{57R45; 57R70}

\section{Introduction}

Topological classif\/ication of spaces with a def\/ined structure on them is one of the main problems of topology. The most fundamental works devoted to functions classif\/ication are papers by G.~Reeb~\cite{Reeb}, A.S.~Kronrod \cite{Kronrod}, A.O.~Prishlyak \cite{Prishlyak}, V.V.~Sharko~\cite{Sharko}, A.A.~Kadubovskyi~\cite{Kadubovskyi2015} and I.A.~Iurchuk~\cite{Iurchuk}. A.S.~Kronrod and G.~Reeb have constructed the graph which allows to classify simple functions on closed surfaces. O.V.~Bolsinov and A.T.~Fomenko~\cite{Bolsinov/Fomenko} also investigate the layer and layer equipped equivalence using the concept of atoms and $f$-atoms.

Currently, there are many papers which are devoted to exploration of topological properties of functions on surfaces with the boundary, e.g., \cite{Hladysh/Prishlyak,Lukova/PrishlyakA/PrishlyakK,Maksumenko/Polulyakh,Vyatchaninova}. In~\cite{Hladysh/Prishlyak} we consider Morse functions, whose critical points belong to the boundary and are non-degenerated critical points of restriction of the function to the boundary of the manifold (called here $mm$\emph{-functions}).

According to \cite[Theorem~1.3]{Prishlyak2002} the function with isolated critical points on a closed surface is locally topologically equivalent to the function $f(x,y)=\operatorname{Re}(x+iy)^{k}$ for some integer non\-ne\-ga\-tive~$k$. In this paper we generalize this theorem up to the case when a critical point belongs to the boundary of the surface.

Another kind of important problem is to f\/ind a minimum number of critical points of functions def\/ined on a f\/ixed surface. This problem is considered in details in~\cite{Hladysh/Prishlyak,Lukova}. We call the function, having the minimum number of critical points as \emph{optimal}. As well as in~\cite{Hladysh/Prishlyak} the criterion of optimality of a $mm$-function is proved.

The present paper examines the class of smooth functions with isolated critical points on a~manifold $M$ belonging to the boundary and being also isolated critical points of their restrictions to~$\partial M$. It is also focused on the problem of getting a local topological classif\/ication and f\/inding a criterion of optimality of previously described functions.

\section{Local topological classif\/ication}

Let $M$ be a compact surface with the boundary $\partial M$ and $f\colon M \rightarrow \mathbb{R}$ be a smooth function on~$M$ with a f\/inite number of critical points all of which belong to~$\partial M$. The f\/inite number of critical points is equivalent to their isolation due to the compactness of~$M$.

We denote the restriction of function $f$ to the boundary $\partial M$ of the surface~$M$ by~$f|_{\partial M}$ and the set of smooth functions def\/ined on $M$, whose critical points are isolated, belong to the boundary and are also isolated critical points of restriction to~$\partial M$ by~$\Omega(M)$. Thus, we consider the following class of functions
\begin{gather*}
\Omega(M)=\big\{f\colon M\rightarrow\mathbb{R}\,|\,f\in C^{\infty}(M), \, {\rm CP}(f)={\rm ICP}(f)={\rm ICP}(f|_{\partial M})\big\},
\end{gather*}
where ${\rm CP}(f)$ (${\rm ICP}(f)$) is the set of (isolated) critical points of function~$f$.

\begin{Lemma}\label{lemma2.1}
Let $y$ be a regular value of a smooth function $f\colon M\rightarrow\mathbb{R}$, where $M$ is a compact surface with the boundary. Then the following statements hold true:
\begin{enumerate}\itemsep=0pt
\item[$(i)$] the level $f^{-1}(y)$ consists of a finite number of circles and line segments;

\item[$(ii)$] for arbitrary open set $U$ with smooth boundary being transversal to $f^{-1}(y)$, the intersection $f^{-1}(y)\cap U$ has a finite number of components.
\end{enumerate}
\end{Lemma}

\begin{proof} Firstly, according to \cite[Lemma~1]{Milnor}, the preimage of a regular level is 1-dimensional manifold. Thus, it can not include the isolated points. Also the limit of critical points is a~critical point. It follows from the continuity of partial derivatives and from the def\/inition of critical points.

Let $y$ be a regular value of $f$. Firstly, we prove that level $f^{-1}(y)$ has a f\/inite number of line segments. Let us suppose that it doesn't hold. Then there exists the component of the boundary which intersects the level line of $f$ in the inf\/inite number of points. Whereas the function acquires the same values, then (according to Rollya's theorem) there exists the critical value of function restriction to the boundary between every two such points. Thus, the restriction $f|_{\partial M}$ has the inf\/inite number of critical points. It means that the sequence of critical points has the limit point belonging to level $f^{-1}(y)$. The last statement contradicts to the condition as~$y$ is regular value of~$f$.

In the next part of this proof we are going to show that the level line $f^{-1}(y)$ contains a~f\/inite number of the circles. In order to do this we glue all components of the boundary $\partial M$ by 2-di\-men\-sional disks $D^{2}$. Thus, we get a closed surface $M'$. Let us continue the function $f$ on these disks to smooth function $F$ with a f\/inite number of critical points on each disk (we can do it by arbitrary function and after that approximate it by Morse function). In such a way we get the surface $M'$ and the function $F$ def\/ined on $M'$, whose critical points are also isolated. Then the preimage $F^{-1}(y)$ is compact, because of the compactness of the surface $M'$ and the closedness of the set $F^{-1}(y)$. In what follows $F^{-1}(y)$ includes the f\/inite number of circles. As a~result, the level line $f^{-1}(y)$ also includes the f\/inite number of the circles for the initial surface~$M$, because the number of the circles can only decrease after rejection of glued disks.
\end{proof}

The function with isolated critical point, being not local extreme, on a closed surface is locally topologically equivalent with the following function: $f(x,y)=\operatorname{Re}(x+iy)^{k}$ for some integer $k$, $k\geq1$ \cite{Prishlyak2002}. Therefore, f\/irstly, we consider the level lines of function $f(x,y)=\operatorname{Re}(x+iy)^{k}$, def\/ined on a surface $\mathbb{R}^2_{+}=\{(x,y)\in\mathbb{R}^2\,|\,y\geq0\}$ for $k\in\{1,2,3,4\}$ and in a general case. Also note that $p_0=(0,0)$ is an isolated critical point of the function $\operatorname{Re}(x+iy)^k$ and $0$ is a correspondent critical value.

We say that a function~$f$ has a \emph{local topological presentation} $f(x,y)=\operatorname{Re}(x+iy)^k$, $y\geq0$, if it is locally topologically equivalent to the function $\operatorname{Re}(x+iy)^k$, $y\geq0$.

If $k=1$, then $f$ has a local topological presentation $f(x,y)=x$, $y\geq0$. The level lines of $f$ are shown in Fig.~\ref{fig1}.1. Evidently the level $f(x,y)=0$ is the ray $x=0$, $y\geq0$.

In case $k=2$ we have $f(x,y)=\operatorname{Re}(x+iy)^{2}=x^{2}-y^{2}$, $y\geq0$ and level $f(x,y)=0$ consists of two rays $y=x$, $y\geq0$ and $y=-x$, $y\geq0$, see Fig.~\ref{fig1}.2.

When $k=3$, the function has a local presentation $f(x,y)=\operatorname{Re}(x+iy)^{3}=x^{3}-3xy^{2}$, $y\geq0$. Its level lines are shown in Fig.~\ref{fig1}.3. The critical level $f(x,y)=0$ consists now of three rays $x=0$, $y\geq0$, $y=\frac{x}{\sqrt{3}}$, $y\geq0$, and $y=-\frac{x}{\sqrt{3}}$, $y\geq0$.

If $k=4$, then $f(x,y)=\operatorname{Re}(x+iy)^{4}=x^{4}-4x^{2}y^{2}+y^{4}$, $y\geq0$ and its level set $f(x,y)=0$ consists of four rays $y=x\sqrt{2+\sqrt{3}}$, $y\geq0$, $y=x\sqrt{2-\sqrt{3}}$, $y\geq0$, $y=-x\sqrt{2+\sqrt{3}}$, $y\geq0$, and $y=-x\sqrt{2-\sqrt{3}}$, $y\geq0$, see Fig.~\ref{fig1}.4.
\begin{figure}[t]\centering
\includegraphics{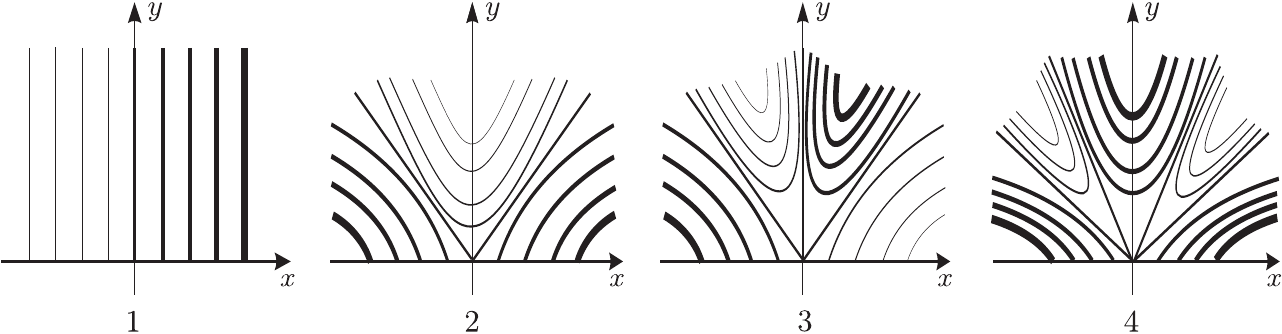}
\caption{}\label{fig1}
\end{figure}

Thus, in all the cases $k\in\{1,2,3,4\}$ the line $y=0$ is the axis of symmetry and it intersects with level $f^{-1}(0)$ of function $f$ at only one point $p_0$. Further we generalize this result to an odd and even $k$.

If $k=2n+1$ (for some integer $n$, $n\geq1$), then function $f$ has the following local topological presentation:
\begin{gather*} f(x,y)=\operatorname{Re}(x+iy)^{2n+1}=
\sum_{j=0}^{n}(-1)^{j}C_{2n+1}^{2j}x^{2n+1-2j}y^{2j}, \qquad y\geq0.
\end{gather*}

Finally, in case $k=2n$ (for some integer $n$, $n\geq1$) we get the function:
\begin{gather*}
f(x,y)=\operatorname{Re}(x+iy)^{2n}=\sum_{j=0}^{n}(-1)^{j}C_{2n}^{2j}x^{2n-2j}y^{2j}, \qquad y\geq0.
\end{gather*}

In both cases the line $y=0$ is the axis of symmetry, because if a point $(x_{0},y_{0})$ belongs to the level $f^{-1}(0)$, then the point $(x_{0},-y_{0})$ also belongs to this level.

\begin{Lemma}\label{lemma2.2}
Let $p_{0}\in\partial M$ be a critical point of a function $f\in\Omega(M)$ with the correspondent critical value $f(p_{0})=0$. Then there exists a neighborhood $U(p_{0})$ and a homeomorphism:
\begin{gather*}
h\colon \ f^{-1}(0)\cap \operatorname{cl}(U(p_{0}))\rightarrow \operatorname{Con}\left(\bigcup_{i=1}^{k}\{x_{i}\}\right)
\end{gather*}
for a finite set of points $\{x_{i},\, i=\overline{1,k}\}$ and some $k \in {\mathbb Z}$, $k\geq1$.

Here $\operatorname{cl}(U(p_{0}))$ is a closure of neighborhood $U(p_{0})$, $\operatorname{Con}\big(\bigcup_{i=1}^{k}\{x_{i}\}\big)$ is a union of $k$ direct lines having a single common point, being their common end $($see Fig.~{\rm \ref{fig2})}.
\end{Lemma}

\begin{figure}[t]\centering
\includegraphics[width=50mm]{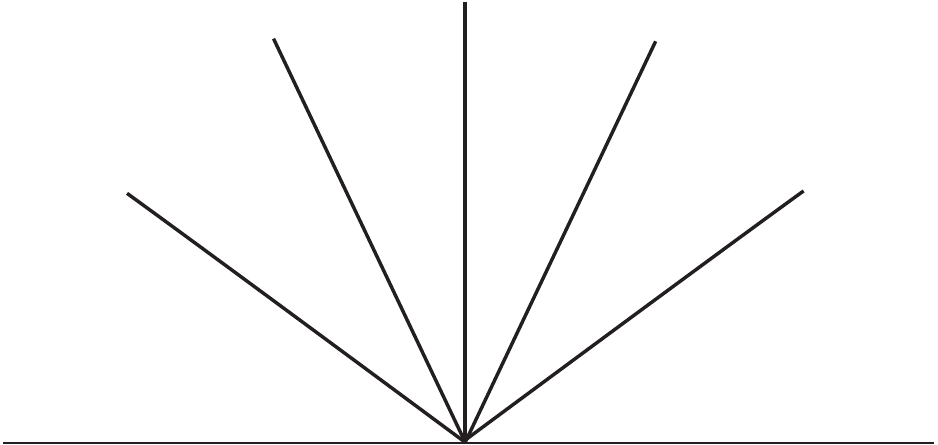}
\caption{}\label{fig2}
\end{figure}

\begin{proof} We consider a neighborhood $U$ of the point $p_{0}$, that does not contain critical points of function $f$ with the exception of $p_0$. Then the set $K := (f^{-1}(0){\setminus} \{p_{0}\})\cap U$ does not contain critical points of $f$ and it is a 1-dimensional manifold. If~$K$ includes a closed curve, then within this curve the function $f$ will have a point of local extremum, which is impossible. In the same way you can make sure that component $K$ can not form the loop with the vertex~$p_{0}$.

Let us show that $K$ has a f\/inite number of components. Suppose there are inf\/initely many such components. One can always choose a neighborhood $U$, that has smooth and transversal to each component of $K$ boundary. Then the set $K\cap \partial U$ has a limit point belonging to the set~$K$. The partial derivatives of~$f$ by the directions $\partial U$ and $K$ equal zero at this limit point. It means that this point is critical. The last sentence contradicts to the conditions. Thus, we proved the f\/initeness of components of the set~$K$.

Afterwards, if it is necessary, again reduce the neighborhood $U$ to one, that does not include that components of $K$ for which $p_{0}$ is not a limit point. Then, the rest of components with the point $p_{0}$ form the union of $k$ direct lines, which have a single point $p_0$ being their common end.
\end{proof}

\begin{Definition}\label{definition2.1}
Two given smooth functions $f$ and $g$, def\/ined on some surfaces $M$ and $N$ respectively, are said to be \emph{layer equivalent} if there exists a~homeomorphism $\lambda \colon M\rightarrow N$, which maps the components of the level sets of $f$ onto the components of the level sets of~$g$.
\end{Definition}

\begin{Definition}\label{definition2.2}
Two given smooth functions $f$ and $g$ are said to be \emph{layer equipped equivalent} in some neighborhoods of their critical levels $f^{-1}(c_{1})$ and $g^{-1}(c_{2})$ if there exist $\varepsilon_{1}>0$, $\varepsilon_{2}>0$ and a homeomorphism $\lambda \colon f^{-1}(c_{1}-\varepsilon_{1},c_{1}+\varepsilon_{1}) \rightarrow g^{-1}(c_{2}-\varepsilon_{2},c_{2}+\varepsilon_{2})$, which maps the components of the level sets of $f$ onto the components of the level sets of~$g$ and $\lambda$ preserve the growing directions of functions.
\end{Definition}

\begin{Definition}\label{definition2.3}
Two given smooth functions $f$ and $g$, def\/ined on some surfaces $M$ and $N$ respectively, are said to be \emph{topologically equivalent} if there exist homeomorphisms $h_1\colon M\rightarrow N$, $h_2\colon \mathbb{R}\rightarrow\mathbb{R}$, such that $h_2\circ f=g\circ h_1$ and $h_2$ preserves the orientation of $\mathbb{R}$.
\end{Definition}

\begin{Theorem}\label{theorem2.1}
Function $f\in\Omega(M)$ is topologically equivalent to the function $g(x,y)=x^{2}+y^{2}$, $y\geq0$ $(g(x,y)=-x^{2}-y^{2}$, $y\geq0)$ in some neighborhood of its local minimum $($maximum$)$ point.
\end{Theorem}

\begin{proof}
Theorem statement has local nature, so we consider some small enough neighborhood of a local minimum (maximum) point and all consideration will take place in this neighborhood.

Every level line of a function $f\in\Omega(M)$ is transversal to the boundary~$\partial M$, because the restriction of this function has no critical point with the exception of minimum (maximum) point. We construct a grad-like vector f\/ield $X$, being tangent to the boundary. Let~$f$ takes value~$c$ at~$p_{0}$. Consider the homeomorphism~$h$ which maps the level $f^{-1}(c+\varepsilon)$ (for some $\varepsilon>0$) into level $x^{2}+y^{2}=\varepsilon$, $y\geq0$ of $g$.

Let us denote the trajectory of f\/ield $X$ which passes though a point $x$ by $\gamma(x)$ and the trajectory of f\/ield $X$ which passes though a~point~$y$ by~$\beta(y)$. Then, desired homeomorphism of neighborhood is def\/ined by the formula: $H(x)=\beta(h(\gamma(x)\cap f^{-1}(c+\varepsilon)))\cap g^{-1}(\varepsilon)$.
\end{proof}

Notice that a critical point is a \emph{saddle critical point} if it is not the point of local minimum or local maximum.

\begin{Theorem}\label{theorem2.2}
Let $f\in\Omega(M)$, $p_{0}$ be a saddle critical point of function~$f$. Then there exists a neighborhood $U(p_{0})$ of~$p_{0}$, such that the restriction $f|_{U(p_0)}$ is topologically equivalent to the function $g(x,y) = \mathrm{Re}(x+iy)^{k}$, $y\geq 0$ which is defined in some neighborhood of $(0,0)$ for some integer $k\geq1$.
\end{Theorem}

\begin{proof}
Let $U(p_{0})$ be the neighborhood with the properties described in Lemma~\ref{lemma2.2}. Then critical level of~$f$ divides~$U(p_{0})$ into regions. We consider one of these regions and denote it by~$V$. Let us suppose that $f(x)<0$ on $V$. Then a vector f\/ield $\operatorname{grad}(f)$ is directed inside the region~$V$ at the points of the intersection $\partial U(p_{0})\cap \partial V$.

We denote the trajectories of gradient vector f\/ield, passing through point $x$ by $\phi_x$. Let us consider a map $h\colon \partial U(p_{0})\cap \partial V\rightarrow f^{-1}(0)\cap\partial V$ being def\/ined by the condition $h(x)=\phi_x\cap f^{-1}(0)$ if $\phi_x\cap f^{-1}(0)\neq\varnothing$ and $h(x)=\lim\limits_{t\to \infty}\phi_x(t)=p_0$ if $\phi_x\cap f^{-1}(0)=\varnothing$ (see Fig.~\ref{fig3}). The map $h$ is continuous, because of the continuous dependence of the solution of dif\/ferential equation on initial condition. Then there exists a point being mapped into the point~$p_0$ and it means that there exists the trajectory $\gamma$ passing through $\partial U(p_{0})\cap \partial V$ and f\/inishing at the point~$p_{0}$. If~$V$ has common points with the boundary, then we change the f\/ield $\operatorname{grad}(f)$ in the neighborhood of the boundary into a such one, which is tangent to the boundary in some neighborhood. Then, the trajectory~$\gamma$ contains into the boundary.
\begin{figure}[t]\centering
\includegraphics{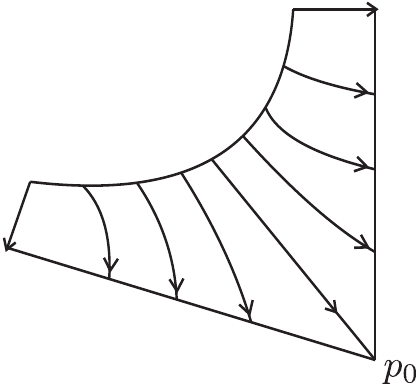}
\caption{}\label{fig3}
\end{figure}

The set $K$ with such constructed trajectories $\gamma$ divides the neighborhood $U(p_{0})$ into regions. Similarly to Theorem~\ref{theorem2.1}, we can construct a homeomorphism of each of these sectors (if it is necessary, after some reduction of $U(p_{0})$) into correspondent sectors of function $g(x,y) = \mathrm{Re}(x+iy)^{k} $, $y\geq 0$ which maps the level lines into level lines. Also if critical point has only one trajectory which enter in it, then the trajectories are mapped into trajectories by this homeomorphism.

Homeomorphisms mentioned above are the same at the common boundary, that is why they determine the searching one.
\end{proof}

Note that in case $k=1$ the neighborhood $U(p_{0})$ has two sectors having common points with the boundary with the exception of $p_{0}$, see Fig.~\ref{fig1}.1.

\section{Atoms of the function}

\begin{Definition}\label{definition3.1}
A function which has at most one critical point at each level line will be called \emph{simple}.
\end{Definition}

\begin{Definition}\label{definition3.2}
An \emph{atom} is a class of layer equivalence of function $f$ restriction to the set $f^{-1}([c-\varepsilon, c+\varepsilon])$, where~$c$ is a critical value of~$f$, for small enough $\varepsilon$, such that the line segment $[c-\varepsilon, c+\varepsilon]$ does not include critical values with the exception of~$c$.
\end{Definition}

\begin{Definition}\label{definition3.3}
A $f$-\emph{atom} is a class of layer equipped equivalence of restriction of function $f$ to the set $f^{-1}([c-\varepsilon, c+\varepsilon])$, where $c$ is critical value of $f$, for small enough $\varepsilon$ as above.
\end{Definition}

Remark that every atom has corespondent two f-atoms, which can be obtained one from another by changing the sing of the function.
In what follows we consider only simple functions and suppose that~$f(p_0)=0$, where~$p_0$ is an isolated saddle critical point of functions~$f$ and~$f|_{\partial M}$.

Let us consider the neighborhood of a critical point $p_{0}$ bounded by $f^{-1}(-\varepsilon)$, $f^{-1}(\varepsilon)$ for some small enough $\varepsilon>0$, by some trajectories of a gradient f\/ield and by the boundary~$\partial M$. The parts of the surface where $f>0$ and $f<0$ will be called the positive and negative sectors of function $f$. We depict these sectors by shaded and unshaded ones. The obtained surface has the structure of $(2k+2)$-gon. If we extend this neighborhood to the neighborhood of a critical level, we get the neighborhood which is homeomorphic to a polygon with glued sides by linear homeomorphism (e.g., in Fig.~\ref{fig4} the side $CB$ is glued with the side $DE$).
\begin{figure}[t]\centering
\includegraphics{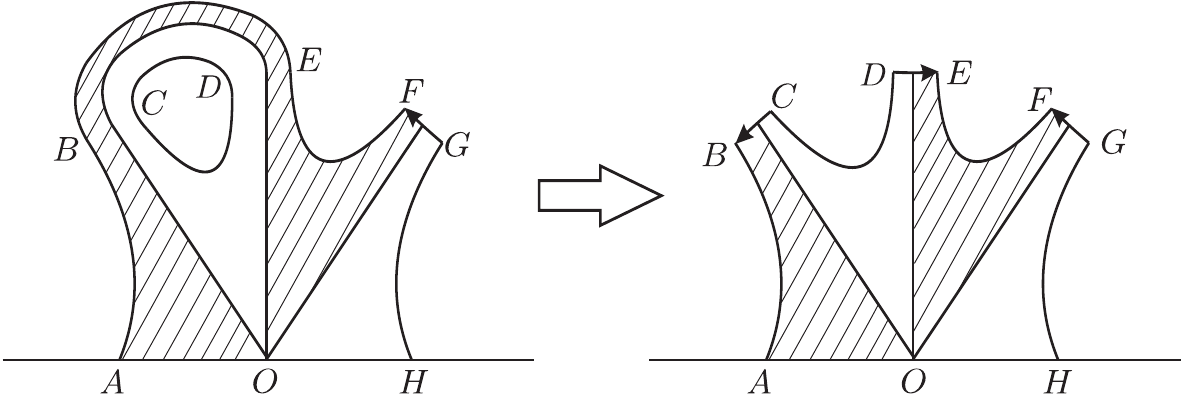}
\caption{}\label{fig4}
\end{figure}

Thus, atom has the structure of $(2k+2)$-gon, which is presented in Fig.~\ref{fig5}.
\begin{figure}[t]\centering
\includegraphics{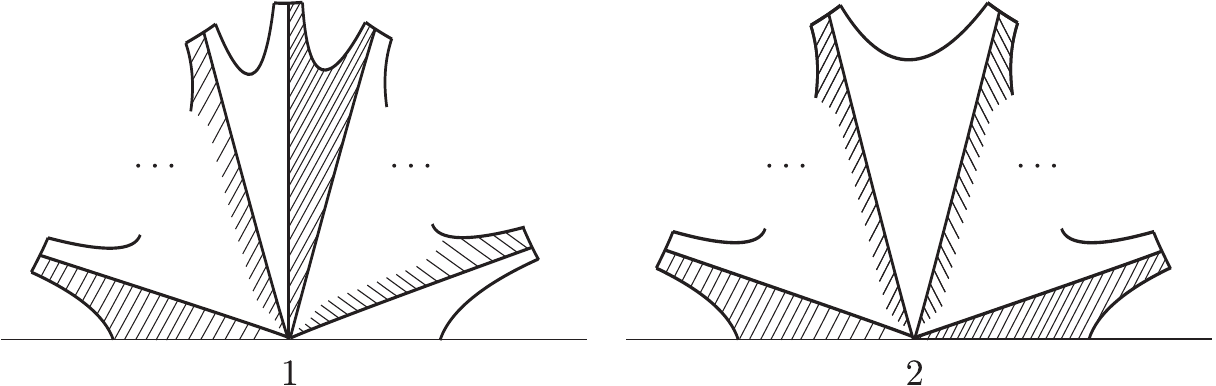}
\caption{}\label{fig5}
\end{figure}

We put a circle with matched points corresponding to the previously described polygon. This circle is the boundary of $(2k+2)$-gon and matched points are the points on the circle belonging to the intersection of shaded and unshaded sectors (in other words, matched points belong to the critical level).

We connect two matched points by a chord if and only if correspondent sides of polygon become glued after extending of critical level neighborhood. In what follows we get the circle with the matched points.

Further f\/ix the orientation on the boundary to numerate the matched points on the circle. If we change the orientation, we get the equivalent atom. Then we numerate matched points in the following way: a point corresponding to a critical one $p_{0}$ we denote by $Q_{0}$, and the rest of points we numerate according to the orientation of the boundary beginning with~$Q_{1}$ up to $Q_{k}$ and point $Q_{0}$ we consider as the point of reference. These points divide the circle into $k+1$ black (thick) and grey (thin) arcs. These arcs correspond to positive and negative sections (see Fig.~\ref{fig6}).
\begin{figure}[t]\centering
\includegraphics{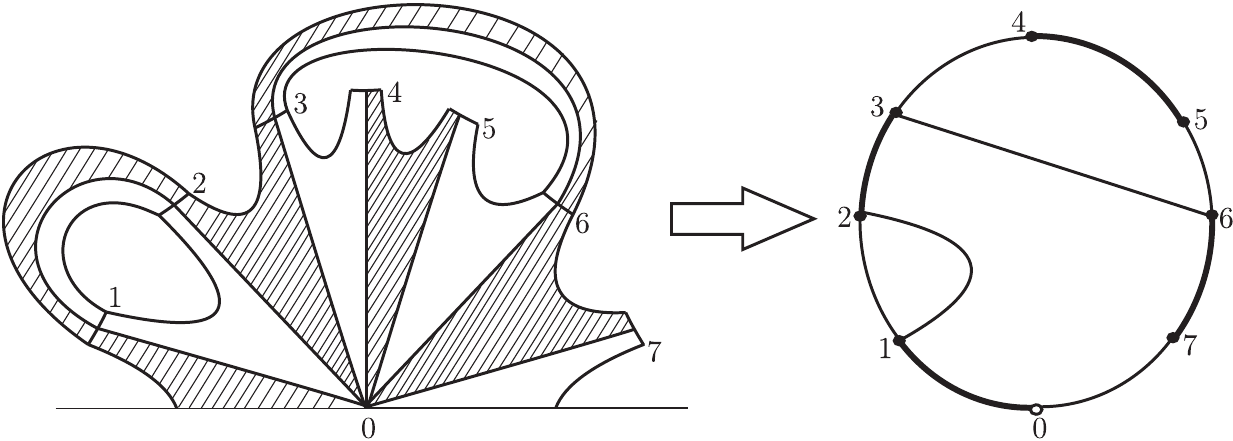}
\caption{}\label{fig6}
\end{figure}

Then, every atom can be def\/ined by the circle with $k+1$ matched points and $l$ chords (for some $l\in\{0,1,2,\ldots,[\frac{k}{2}]\}$). Also one matched point corresponds to a critical point.

\begin{Definition}\label{definition3.4}
\emph{A chord diagram of a saddle critical level} of the function def\/ined on a smooth compact surface with the boundary is the circle with the following elements:
\begin{itemize}\itemsep=0pt
\item[$(1)$] matched points, which are enumerated;
\item[$(2)$] chords, the ends of which are dif\/ferent matched points;
\item[$(3)$] coloration of arcs such that each two neighbor arcs with the exception of arcs near point $Q_{0}$, are of dif\/ferent colors.
\end{itemize}
\end{Definition}

Note that chord diagram def\/ines $f$-atom and if we consider only elements~(1) and~(2) then chord diagram def\/ines atom.

\begin{Definition}\label{definition3.5}
Two chord diagrams are called \emph{equivalent} if they can be obtained one from another by turn or symmetry preserving the elements~(1)--(3).
\end{Definition}

\begin{Definition}\label{definition3.6}
\emph{A free matched point} on chord diagram is the one which is not connected with another matched points by chord.
\end{Definition}

Chord diagrams are also considered in papers \cite{Kadubovskyi2015, Khruzin, Stoimenov2000}.

The circle of a chord diagram we denote by $S^{1}$, matched points by $0,1,\ldots,k$, and chord which connects points $i$ and $j$ by $l_{ij}$, $i,j\in\{1,\ldots,k\}$. We will say that $j$ is a number of a~matched point~$Q_{j}$. Free points with the exception of $Q_{0}$ we denote by~$Q_{i^{*}}$, where $i$ is the number of correspondent matched points. Each matched point of chord diagram corresponds to two vertices of the (previously described) $(2k+2)$-gon and one of these points belongs to a~positive sector~($f>0$) and another one to a negative sector ($f<0$). That is why we denote these points by~$P_{i}$ (positive) and $N_{i}$ (negative), where $i$ is the number of matched point~$Q_{i}$.

\begin{Lemma}\label{lemma3.1}
A number of free points $N_{\rm f.p}$ can be calculated from the formula
\begin{gather*}
N_{\rm f.p}=k-2\cdot N_{\rm ch}+1,
\end{gather*}
where $N_{\rm ch}$ is a number of chords of chord diagram.
\end{Lemma}

\begin{proof}
There exists a neighborhood of critical level which can be represented in the form of polygon with $(2k+2)$ vertices and $2l$ glued sides (for some $l$). Then, the chord diagram of this critical level includes $k+1$ matched points and $l$ chords. Thus, after the renotation $N_{\rm ch}:=l$ we get the formula of calculation of free points number.
\end{proof}

In what follows we construct a substitution from the chord diagram. The substitution includes the cycle $(ij)$ (for some $i,j\in\{1,2,\ldots,k\}$) if and only if matched points $Q_i$ and $Q_j$ are connected by a chord. This substitution includes only the cycles with 2 elements.

The substitutions constructed in previously described way will be called \emph{gluing substitution} and the atom with its gluing substitution ${\tau^{(k)}}$ on the set $\{1,2,\ldots,k\}$ will be denoted $A_{\tau^{(k)}}$.

\begin{Theorem}\label{theorem3.1}
The following statements hold true:
\begin{itemize}\itemsep=0pt
\item[$1)$] each atom of saddle critical level coincides with atom $A_{\tau^{(k)}}$ for some gluing substitution $\tau^{(k)}$ on the set $\{1,2,\ldots,k\}$ and this substitution defines the gluing of atom sides;
\item[$2)$] a number $N_{k}$ of atoms $A_{\tau^{(k)}}$ can be calculated by the formula:
\begin{gather*}
N_{1}=1, \qquad N_{2}=2, \qquad N_{3}=4, \qquad N_{k}=2\sum_{j=1}^{k-3}P_{j}^{(k)}+P_{k-2}^{(k)},
\end{gather*}
where for each $k$ the set of numbers $P_{j}^{(k)}$ defined by a recurrent correlation
\begin{gather*}
P_{0}^{(k)}=1, \qquad P_{1}^{(k)}=k-1, \qquad P_{2}^{(k)}=k-2, \\
P_{j}^{(k)}=\big(P_{0}^{(k)}+P_{1}^{(k)}+\cdots+P_{j-2}^{(k)}\big)(k-j), \qquad j\in {\mathbb Z}, \quad j\geq3.
\end{gather*}
\end{itemize}
\end{Theorem}

\begin{proof}
1) Let us consider a neighborhood of a saddle critical level correspondent to critical point~$p_{0}$, being presented in the form of polygon with $(2k+2)$ vertices with glued (or unglued) sides. We construct a critical level chord diagram, using this polygon. Thus each possible gluing of sides of the polygon corresponds to some substitution $\tau$, def\/ined on the set $\{1,2,\dots,k\}$. Also it should be mentioned that if there is no side being connected with another one we get the trivial substitution.

2) A number of atoms is equal to a number of non-equivalent chord diagram. That is why the number of non-equivalent chord diagram with $k+1$ matched points and ``free'' number of chords (from~$0$ to $[\frac{k}{2}]$) we also denote by~$N_{k}$.

Hence, we can connect point $Q_1$ with another one by a chord or not. The f\/irst action can be made by $k-1$ ways and the number of atoms is equal to $N_{k-2}$. In the second case we get $k-1$ matched points and the number of possible gluings equals $N_{k-1}$. Thus, we have the recurrent formula: $N_{k}=N_{k-1}+(k-1) N_{k-2}$, $g\geq3$.

Let us use the same arguments again and again to see the regularity:
\begin{gather*}
N_{k} = N_{k-1}+(k-1) N_{k} = (1+k-1) N_{k-2}+(k-2) N_{k-3} \\
\hphantom{N_{k}}{} = (1+k-1+k-2) N_{k-3}+(1+k-1)(k-3) N_{k-4}\\
\hphantom{N_{k}}{} =(1+k-1+k-2+(1+k-1)(k-3)) N_{k-4}+(1+k-1+k-2)(k-4)N_{k-5} \\
\hphantom{N_{k}}{} = (1+k-1+k-2+(1+k-1)(k-3)+(1+k-1+k-2)(k-4)) N_{k-5}\\
\hphantom{N_{k}=}{}
+(1+k-1+k-2+(1+k-1)(k-3))(k-5) N_{k-6} = \cdots \\
\hphantom{N_{k}}{} = ((1+k-1+k-2+(1+k-1)(k-3)+(1+k-1+k-2)(k-4)\\
\hphantom{N_{k}=}{}
+(1+k-1+k-2+(1+k-1)(k-3))(k-5)+\cdots\\
\hphantom{N_{k}=}{} +(1+k-1+k-2+(1+k-1)(k-3)+\cdots\\
\hphantom{N_{k}=}{}
+(1+k-1+k-2+\cdots+ (k-(k-2)+2)(k-(k-2)+1))) N_{k-(k-2)}\\
\hphantom{N_{k}=}{}
 + ((1+k-1+k-2+(1+k-1)(k-3)+(1+k-1+k-2)(k-4)\\
\hphantom{N_{k}=}{}
+(1+k-1+k-2+(1+k-1)(k-3))(k-5)+\cdots \\
\hphantom{N_{k}=}{}
+(1+k-1+k-2+(1+k-1)(k-3)+\cdots\\
\hphantom{N_{k}=}{}
+(1+k-1+k-2+\cdots+ (k-(k-1)+3))((k-(k-1)+1))) N_{k-(k-1)}.
\end{gather*}

Consider the notations
\begin{gather*}
P_{0}^{(k)}=1, \qquad P_{1}^{(k)}=k-1, \qquad P_{2}^{(k)}=k-2,\\
P_{j}^{(k)}=\big(P_{0}^{(k)}+P_{1}^{(k)}+\cdots+P_{j-2}^{(k)}\big)(k-j)
\end{gather*} then the previous sequence of equalities can be rewritten in the form
\begin{gather*}
N_{k}= P_{0}^{(k)} N_{k-1}+P_{1}^{(k)} N_{k-1} = \big(P_{0}^{(k)}+P_{1}^{(k)}\big) N_{k-2}+P_{0}^{(k)}P_{2}^{(k)} N_{k-3}\\
\hphantom{N_{k}}{} = \big(P_{0}^{(k)}+P_{1}^{(k)}+P_{2}^{(k)}\big) N_{k-3}+P_{3}^{(k)} N_{k-4}\\
\hphantom{N_{k}}{} = \big(P_{0}^{(k)}+P_{1}^{(k)}+P_{2}^{(k)}+P_{3}^{(k)}\big) N_{k-4}+P_{4}^{(k)} N_{k-5}\\
\hphantom{N_{k}}{} = \big(P_{0}^{(k)}+P_{1}^{(k)}+P_{2}^{(k)}+P_{3}^{(k)}+P_{4}^{(k)}\big) N_{k-5}+P_{5}^{(k)} N_{k-6}= \cdots\\
\hphantom{N_{k}}{} = \big(P_{0}^{(k)}+P_{1}^{(k)}+\cdots+P_{k-2-1}^{(k)}\big) N_{k-(k-2)}+P_{k-1-1}^{(k)} N_{k-(k-1)}\\
\hphantom{N_{k}}{} = N_{2} \sum_{j=0}^{k-3}P_{j}^{(k)}+N_{1} P_{k-2}^{(k)}= 2 \sum_{j=0}^{k-3}P_{j}^{(k)}+P_{k-2}^{(k)}.\tag*{\qed}
\end{gather*}\renewcommand{\qed}{}
\end{proof}

According to the formula from Theorem~\ref{theorem3.1} we can calculate the number of atoms $A_{\tau^{(k)}}$. The results of such calculations for $k\leq20$ are presented in Table~\ref{table1}.
\begin{table}[h!]\centering
\caption{}\label{table1}\vspace{1mm}

\begin{tabular}{|c|r|c|r|c|r|c|r|}
\hline
$k$ & \multicolumn{1}{|c|}{$N_{k}$} & $k$ & \multicolumn{1}{|c|}{$N_{k}$} & $k$ & \multicolumn{1}{|c|}{$N_{k}$} & $k$ & \multicolumn{1}{|c|}{$N_{k}$} \\
\hline
1 & 1 & \hphantom{1}6 & 76 & 11 & 35696 & 16 & 46206736\\
2 & 2 & \hphantom{1}7 & 232 & 12 & 140152 & 17 & 211799312\\
3 & 4 & \hphantom{1}8 & 764 & 13 & 568504 & 18 & 997313824\\
4 & 10 & \hphantom{1}9 & 2620 & 14 & 2390480 & 19 & 4809701440\\
5 & 26 & 10 & 9496 & 15 & 10349536 & 20 & 23758664096\\
\hline
\end{tabular}
\end{table}

\section{Optimal functions}

Further we consider a smooth surface $M$ with a component of the boundary $\partial M$ and a simple smooth function~$f\in\Omega(M)$.

\begin{Definition}\label{definition4.1}
We call a function $f\in\Omega(M)$ \emph{optimal on the surface}~$M$ if it has the minimum possible number of critical points on $M$ among all functions from $\Omega(M)$.
\end{Definition}

\subsection{Optimality criterion of a function}

\begin{Theorem}\label{theorem4.1}
Let $f\in\Omega(M)$ and $M$ be a connected surface with the connected boundary, being not homeomorphic to a $2$-dimensional disk. Then the function~$f$ is optimal if and only if it has exactly three critical points.
\end{Theorem}

\begin{proof} Firstly, remark that a function with three critical points has one minimum point, one maximum point and the third one is a saddle critical point.

\emph{Necessity.} To prove the theorem we show will the next statements: 1) existence of a smooth function, that has three critical points on the surface~$M$, being not 2-dimensional disk; 2)~if a~function has two critical points on a~surface $N$, where $N$ is a 2-dimensional disk. Then from 1) follows that optimal function has no more than~$3$ critical points.
It means that an optimal function has exactly three critical points on the surface with boundary with the exception of 2-dimensional disk.

1) Firstly, we consider the case of oriented surface. Let $M$ be an oriented surface by genus~$g$ with one component of the boundary, which is obtained by gluing of atom $A_{\tau^{(4g+1)}}$ with the substitution{\samepage
\begin{gather*}\begin{split}
& (1,4g)(2,4g+1)(3,4g-2)(4,4g-1)\cdots(2i-1,4g-2i+2)\\
& \qquad {}\times (2i,4g-2i+3)\cdots(2g-1,2g+2)(2g,2g+3).\end{split}
\end{gather*}
The surface $M$ can be def\/ined by a~chord diagram shown in Fig.~\ref{fig7}.1.}
\begin{figure}[t]\centering
\includegraphics{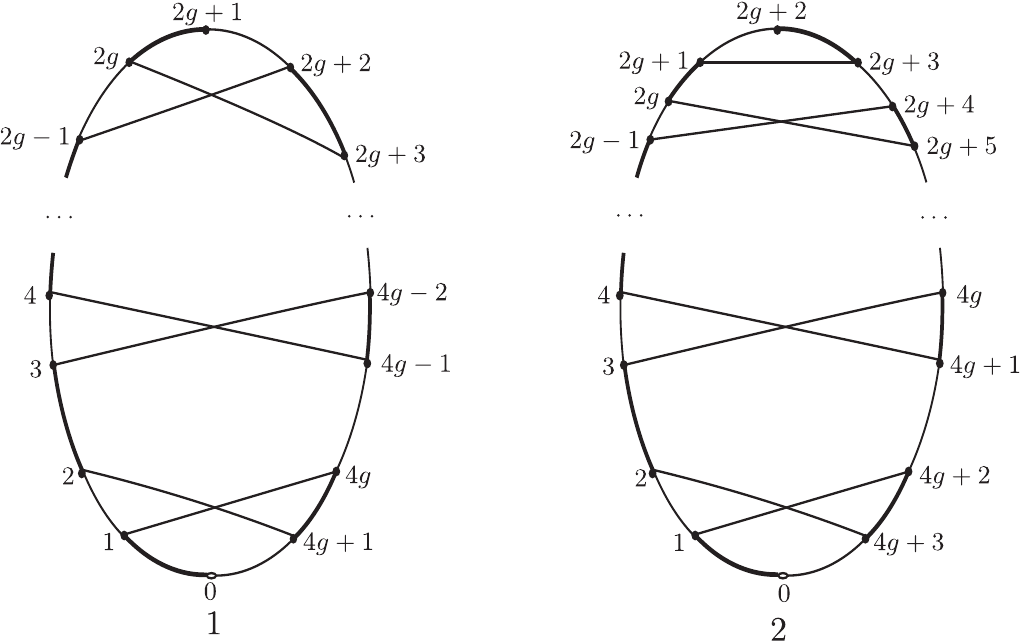}
\caption{}\label{fig7}
\end{figure}

On the other hand, the atom of critical level of function $\operatorname{Re}(x+iy)^{4g+1}$, def\/ined on some non-oriented surface by genus $2g+1$, can be represented in the form of $(4g+2)$-gon and correspondent chord diagram coincides with already described (see Fig.~\ref{fig7}.2) chord diagram.

Let us consider half-disks
\begin{gather*}
D^{2}_{+}=\big\{(x,y)\in{\mathbb R}^{2} \,|\, x^{2}+y^{2}\leq1, \,y\geq0\big\}, \qquad
D^{2}_{-}=\big\{(x,y)\in{\mathbb R}^2 \,|\, x^{2}+y^{2}\leq1,\, y\geq0\big\},
\end{gather*}
homeomorphisms
\begin{gather*}
\widehat{h_{1,2}}\colon \ [0,4g+1]\rightarrow \partial_{\pm}D^{2}_{\pm},
\end{gather*}
where
\begin{gather*}
\partial_{+}D^{2}_{+}=\big\{(x,y)\in\partial D^{2}_{+}\,|\,y>0\big \}, \qquad
\partial_{-}D^{2}_{-}=\big\{(x,y)\in\partial D^{2}_{-}\,|\,y<0\},
\end{gather*}
which map
\begin{gather*}
\widehat{h_{1,2}}(j)=\left(-1+\frac{j}{4g+1},\pm\sqrt{1-(-1+\frac{j}{4g+1})^{2}}\right),\qquad j=\overline{0,4g+1},
\end{gather*}
and embeddings
\begin{gather*}
h_{1,2}\colon \ [0,4g+1] \rightarrow \partial M,
\end{gather*}
such that
\begin{gather*}
h_{1}(0)=P_{0}, \qquad h_{1}(1)=P_{1}, \qquad h_{1}(2)=P_{4g}, \qquad h_{1}(3)=P_{4g+1}, \qquad \ldots, \\
h_{1}(4g)=P_{2g}, \qquad h_{1}(4g+1)=P_{2g+1}, \qquad h_{2}(0)=N_{0}, \qquad h_{2}(1)=N_{4g+1},\\
h_{2}(2)=N_{2}, \qquad h_{2}(3)=N_{1}, \qquad \ldots, \qquad h_{2}(4g)=N_{2g+2}, \qquad h_{2}(4g+1)=N_{2g+1}.
\end{gather*}
Then we clue the obtained atom by half-disks $D^{2}_{\pm}\cup[-1,1]\times\{0\}$ according to the map
\begin{gather*}
\widehat{h}_{1,2} \circ h_{1,2}\colon \ D^{2}_{\pm} \rightarrow \partial M.
\end{gather*}

After this gluing we get a surface with one component of the boundary. Also function $\operatorname{Re}(x+iy)^{4g+1}$ can be continued to the half-disks by its values on the boundary, being used to clue, and, if it is necessary, we can smooth the function (similarly, for example, to work \cite[Section~5]{Konner/Floid}).

A non-oriented surface and a smooth function on it with three isolated critical points can be obtained by using previous considerations with the atom $A_{\tau^{(4g+3)}}$ and the substitution
\begin{gather*}
(1,4g+2)(2,4g+3)(3,4g)(4,4g+1)\cdots(2i-1,4g-2i+4)(2i,4g-2i+5)\cdots\\
\qquad{}\times (2g-1,2g+4)(2g,2g+5)(2g+1,2g+3)
\end{gather*}
(see Fig.~\ref{fig6}.2), which coincide with the atom of saddle critical level of function $\operatorname{Re}(x+iy)^{4g+3}$, and gluing by half-disks $D^{2}_{\pm}$ according to the embeddings
\begin{gather*}
\widehat{h}_{1,2} \circ h_{1,2}\colon \ D^{2}_{\pm} \rightarrow \partial M,
 \end{gather*}
 where
 \begin{gather*}
 \widehat{h_{1,2}}\colon \ [0,4g+1]\rightarrow D^{2}_{\pm}, \\
 \widehat{h_{1,2}}(j)=\left(-1+\frac{j}{4g+3},\pm\sqrt{1-(-1+\frac{j}{4g+3})^{2}}\right), \qquad j=\overline{0,4g+3},
 \end{gather*}
 and
 \begin{gather*}
 h_{1,2}\colon \ [0,4g+3] \rightarrow \partial M,
\qquad h_{1}(0)=P_{0}, \qquad h_{1}(1)=P_{1}, \qquad h_{1}(2)=P_{4g+2}, \\
h_{1}(3)=P_{4g+3}, \qquad \ldots, \qquad h_{1}(4g+1)=P_{2g+1}, \qquad h_{1}(4g+2)=P_{2g+3},\\
 h_{1}(4g+3)=P_{2g+2}, \qquad h_{2}(0)=N_{0}, \qquad h_{2}(1)=N_{4g+3}, \qquad h_{2}(2)=N_{2},\\
 h_{2}(3)=N_{1}, \qquad \ldots, \qquad h_{2}(4g+2)=N_{2g+1}, \qquad h_{2}(4g+3)=N_{2g+2}.
\end{gather*}
Also it should be mentioned that the atom $A_{\tau^{(4g+3)}}$ contains the twisted rectangle which corresponds to chord $(2g+1,2g+3)$.

Thus, we constructed the smooth function with three isolated critical points on the boundary and also are critical points of function restriction to the boundary of the surface.

2) Let function $f$ has two critical points on a surface~$N$. Then we consider a gradient-like vector f\/ield for $f$, being tangent to $N$ (similar to~\cite{Borodzic/Nemeti/Ranicki}). After this we consider the function mapping levels of function $f$ on $N$ into line segments $y={\rm const}$ on~$D^{2}$, and trajectories of gradient-like vector f\/ield are mapped into curves $\gamma_{t}=\{c\cos{t},t(1-c)+\sin{t} \,|\, c\in[0,1]\}$, $t\in[t_{c}^{-},t_{c}^{+}]$, where~$t_{c}^{-}$,~$t_{c}^{+}$ are the smallest modulo roots of equations $t(1-c)+\sin{t}=-1$ and $t(1-c)+\sin{t}=1$ accordingly. This function def\/ines the homeomorphism between $N$ and 2-dimensional disk $D^{2}$, because only one level line and one trajectory path though each point of~$D^{2}$.

\emph{Sufficiency.} Suppose that it doesn't hold. Then there exists a function, having exactly three critical points on a~def\/ined surface but being not optimal. It means that optimal function on this surface has two critical points (because a smooth function on every compact surface has, at least, two critical points). Thus, this surface is a 2-dimensional disk because of the item~2) in necessity part of the proof. In such a way we get the contradiction. The theorem is proved.
\end{proof}

Note that in case of 2-dimensional disk an optimal function has two critical point and can be realized by a height function.

Further we suppose that optimal function has critical values equal $-1$, $0$, $1$ (we can do it, because there exists a homeomorphism of straight line, mapping three critical values of optimal function into points $-1$, $0$, $1$).

\subsection{The case of oriented surface with one component of the boundary}

\begin{Definition}\label{definition4.2}
A homeomorphism $h\colon [0,k] \rightarrow S^{1}\cup \operatorname{Int}\{l_{mn}\,|\,m,n\in\{1,\ldots,k\}\}$ will be called a \emph{full way} between free points $Q_{0}$ and $Q_{i^{*}}$ (for some $i\in\overline{1,k}$) if it satisf\/ies the conditions: $1)$~$h(0)=Q_{0}$, $h(k)=Q_{i^{*}}$; $2)$~$\forall\, t\in\{1,2,\ldots,k-1\}$: $h(t)\in\{Q_{1},Q_{2},\ldots,Q_{i^{*}-1},Q_{i^{*}+1},\ldots,Q_{k}\}$; $3)$~$\forall\, j\in\{1,2,\ldots,k-1\}$ $\forall\, t\in(j,j+1)$: $f(t)$ belongs either to the interior of arc, or to the interior of chord; $4)$~the direction can not be changed during the moving on f\/ixed chord diagram.
\end{Definition}

\begin{Theorem}\label{theorem4.2}
A chord diagram of saddle critical level of optimal function on oriented surface with one component of the boundary satisfies the following conditions:
\begin{itemize}\itemsep=0pt
\item[$1)$] every chord divides the circle into two arcs, each of which contains an even number of matched points;
\item[$2)$] the chord diagram has $k+1=4n+2$ matched points $($for some integer~$n$, $n\geq1)$ and there exist exactly two free points, one of which is $Q_{0}$;
\item[$3)$] there exist two full ways between free points.
\end{itemize}
\end{Theorem}

\begin{proof}
1) The f\/irst item means that we can connect only such matched points of chord diagram which have odd and even numbers. It follows from the orientation of the surface, because in another case corresponding atom includes the M\"{o}bius strip.

2) According to the previous notations, the number of marked points is equal to $k+1$ and the number of chords is equal to~$l$. Let us show that: a)~$k$~is odd ($\Leftrightarrow k+1$ is even); b)~the chord diagram includes two free points; c)~$l=2n$ for some integer nonnegative~$n$. Then from~b) and~c) follows that $k\equiv1$ $({\rm mod}~4)$ (it holds if and only if $k+1=4n+2$).

a) The function $f$ changes the sign when it passes thought critical point because $f(p_{0})=0$. It means that arcs, being adjacent with point $Q_0$, are of deferent color. Thus, the number of matched points is even, because the color of arcs alternates by passing the circle~$S^1$ of chord diagram.

b) Suppose that it doesn't hold. Then the chord diagram has less than $2n$ chords and, at least, 3 free points. It means that appropriate~3 sides of atom belong to the boundary because after extend of the neighborhood of critical level these~3 sides don't clue with another sides. That is why, at least, 4 half-disks~$D^{2}_{\pm}$ become glued to the surface with further increase of the neighborhood. The last sentence infers that this function has at least 5 critical points and as a~result isn't optimal.

Thus, chord diagram of optimal function on the oriented surface with one component of the boundary has 2 free points, one of which $Q_{0}$ corresponds to a critical point.

c) Gluing of each pair of matched points (with odd and even number) is equivalent to attaching of 1-handle (rectangle). Also this gluing increases the number of components of the boundary by~1 (if 1-handle is glued to one component of the boundary) or decreases by~1 (if 1-handle is glued to two components of the boundary) and at the same time the genus of the surface becomes increased after each gluing.

We start from a half-disk $D^{2}_{+}=\{(x,y)\in\mathbb{R}^{2}\,|\, x^{2}+y^{2}\leq1,\, y\geq0\}$ (without glued 1-handles), having one component of the boundary and every glued 1-handle changes the number of components of the boundary by~1. It means that we can get the surface with one component of the boundary only in case of even number of gluings.

3) Let us suppose that there does not exist a full way on a def\/ined chord diagram. Then we get, at least, two components of the boundary of the surface, which contradicts to the conditions of the theorem.
\end{proof}

\begin{Corollary}\label{corollary4.1}
A gluing substitution doesn't include the cycle $(i,j+1)(i+1,j)$ for all possible $i,j\in\{\overline{0,k}\}$.
\end{Corollary}

\begin{proof}
If the gluing substitution includes the cycle $(i,j+1)(i+1,j)$ for some $i$, $j$, then chord diagram doesn't have a~full way, because every path from point $Q_0$ to arbitrary point $Q_p$, $p\in\{0,k\}{\setminus}\{i,i+1,j,j+1\}$ doesn't pass through the arcs $\widehat{i,i+1}$, $\widehat{j,j+1}$ and chords $l_{i,j+1}$, $l_{i+1,j}$.
\end{proof}

\begin{Corollary}\label{corollary4.2}
Every free point of a chord diagram of saddle critical level, except of~$Q_{0}$, has an odd number.
\end{Corollary}

\begin{proof}
From 2) follows that, with the exception of marked point $Q_{0}$, we have matched points $Q_{1},\ldots,Q_{4n+1}$, among which {\it exactly} $2n$ have an even number and $2n+1$ have an odd number. Then another free point has an odd number (the second item of Theorem~\ref{theorem4.2}).
\end{proof}

\begin{Theorem}[criterion of topological equivalence] \label{theorem4.3}
Optimal functions are topologically equi\-va\-lent if and only if their chord diagrams of saddle critical levels are equivalent.
\end{Theorem}

\begin{proof}
\emph{Necessity.} The topological equivalence of optimal functions induces the equivalence of chord diagrams, because of the construction of chord diagrams. Then the equivalence of their chord diagram, what follows from topological equivalence of optimal functions.

\emph{Sufficiency.} Suppose that chord diagrams of optimal functions~$f$ and~$h$ are equivalent. In what follows the existence of homeomorphisms from level lines of functions $f$ and $h$ to level lines of function $\operatorname{Re}(x+iy)^{k}$ (Theorem~\ref{theorem4.1}). Let $\varphi$ and $\psi$ be these homeomorphisms. Then the homeomorphism from some neighborhood of a saddle critical point of function $f$ to correspondent neighborhood of $h$ can be def\/ined by the formula $\varphi^{-1}\circ\psi$.

Let us consider the atom of critical level without a critical point. Then a~chord or a free point (with the exception of $Q_0$) conforms to each connected component. The equivalence of chord diagrams def\/ines the bijection between these level components, that is why, we can continue the homeomorphism between neighborhoods of critical point of functions $f$ and $g$ to homeomorphism of neighborhood of critical level of these functions.

Thus, we continue this homeomorphism to the chords. And we have the point either on chord or on the boundary of $(8n+4)$-gon for each trajectory of gradient-like vector f\/ield. It means that we get the bijection between these trajectories. Thus, the homeomorphism between appropriate trajectories can be constructed and it satisf\/ies the condition that the points with one and the same values map one into another (in other words, this homeomorphism preserves the value of the function). We obtain the homeomorphism, because on every trajectory there exits a single point with a f\/ixed value from $(0,1)$. These homeomorphisms def\/ine the homeomorphism of the surface and it is a~topological equivalence.
\end{proof}

\begin{Theorem}[realization]\label{theorem4.4}
If a chord diagram satisfies the conditions $1)$--$3)$ of Theorem~{\rm \ref{theorem4.2}}, then there exists an optimal function, chord diagram of saddle critical level of which coincides with the first one.
\end{Theorem}

\begin{proof}
A chord diagram can be def\/ined by a number $k$ and by a gluing substitution. That is why, we consider the function $f(x,y)=\operatorname{Re}(x+iy)^{4n+1}$, $y\geq0$, its isolated critical point $p_{0}=(0,0)$ and $\varepsilon$-neighborhood of critical level $f(p_{0})=0$ (for some $\varepsilon>0$), which is presented in the form of $(4n+2)$-gon with colored sections (see Fig.~\ref{fig5}.1). Let us suppose that sides of this polygon are line segments and have the length which equals~$2\varepsilon$. Then enumerate the sides of this polygon from~$S_{0}$ up to~$S_{4n+1}$ as it is shown in Fig.~\ref{fig8}.
\begin{figure}[t]\centering
\includegraphics{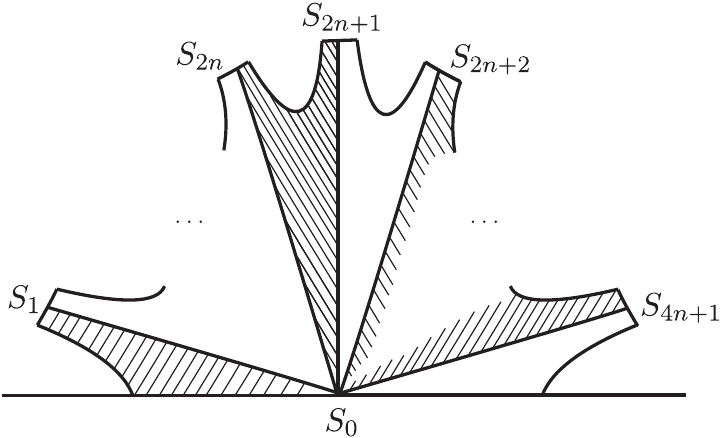}
\caption{}\label{fig8}
\end{figure}

After this we glue the rectangle $[0,1]\times[-\varepsilon,\varepsilon]$ to sides $S_{i}$ and $S_{j}$ (according to the gluing substitution) by the following image: if gluing substitution includes the cycle $(i,j)$, then
\begin{gather*}
p\colon \ \{0,1\}\times[-\varepsilon,\varepsilon]\rightarrow\{S_{i},S_{j}\}, \qquad p((0,-\varepsilon))=Q_{2i}, \qquad p((0,\varepsilon))=Q_{2i-1},\\
p((1,-\varepsilon))=Q_{2j-1}, \qquad p((1,-\varepsilon))=Q_{2j-1}, \qquad p((1,\varepsilon))=Q_{2j}.
\end{gather*}
Def\/ine the function $f(x,y)$ on glued rectangle by its second coordinate \smash{$f(x,y)=y$}. As a result, we get an atom of saddle critical level. Whereafter we glue this atom by half-disks
\begin{gather*}
D_{+}=\big\{(x,y)\in{\mathbb R}^{2}\, |\, x^{2}+(y-\varepsilon)^{2}\leq(1-\varepsilon)^{2},\, y\geq\varepsilon\big\},\\
 D_{-}=\big\{(x,y)\in{\mathbb R}^{2}\, |\, x^{2}+(y+\varepsilon)^{2}\leq(1-\varepsilon)^{2},\, y\leq-\varepsilon\big\},
\end{gather*}
as it was done in Theorem~\ref{theorem4.1} in oriented case and continue function $f$ on this half-disks by the formula \smash{$f(x,y)=y$}. If it is necessary we can smooth the obtained function (similar to work~\cite{Konner/Floid}). Thus, we get the surface and the smooth function which is def\/ined on it, such that chord diagram of saddle critical level of this function coincides with the initial chord diagram.
\end{proof}

\subsection{Examples of calculations}

From Theorems~\ref{theorem4.2}, \ref{theorem4.3} and~\ref{theorem4.4} follows that every chord diagram def\/ines (up to topological equivalence) an optimal function on oriented surface with one component of the boundary if and only it satisf\/ies the conditions 1)--3) of Theorem~\ref{theorem4.2}.

Using this properties of chord diagrams, there was calculated the number of optimal layer (topological) non-equivalent functions def\/ined on oriented surface with one component of the boundary in cases:
\begin{itemize}\itemsep=0pt
\item[$1)$] surface by genus 1: 1 (1):
\begin{center}
\includegraphics{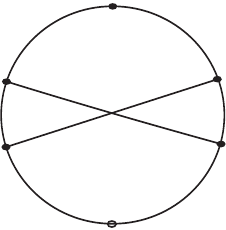}
\end{center}

\item[$2)$] surface by genus 2: 5 (8):{\samepage
\begin{center}
\includegraphics{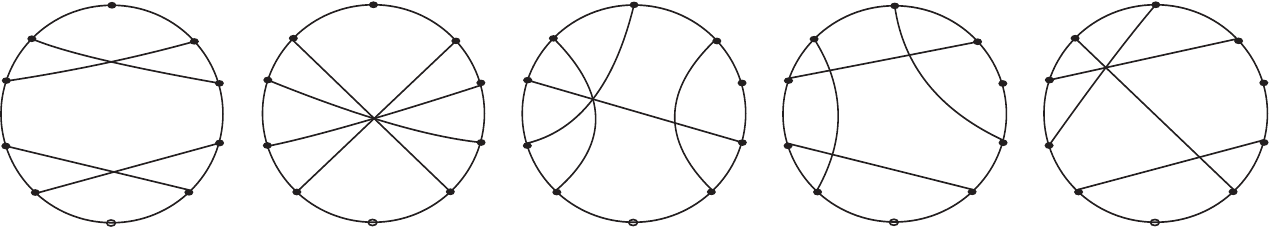}
\end{center}}

\item[$3)$] surface by genus 3: 94 (182).
\end{itemize}

Also in case of oriented surface by genus~3 we used chord diagrams with 7 chords of closed surface (chords 1--25), which are described in detail in paper~\cite{Kadubovskyi}.
\begin{center}
\includegraphics{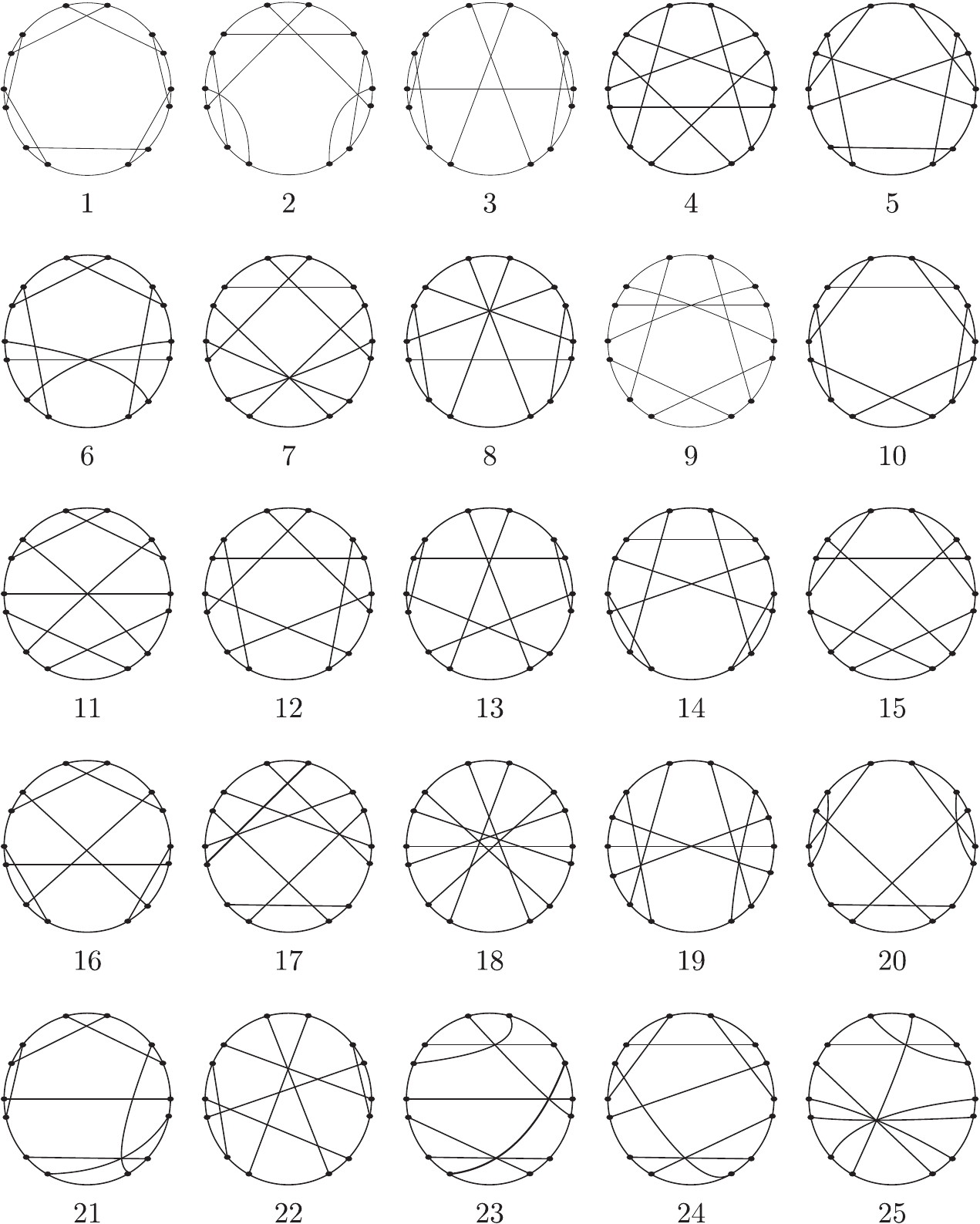}
\end{center}

We take into account possible symmetries and rotations of these chord diagrams and in such a way we count the number of nonequivalent atoms and $f$-atoms of saddle critical level. We use the abbreviation c.d.\ for chord diagrams.

Thus, for example, we describe how many atoms we get from each c.d.\ 1--25: from c.d.~1 we get 1 atom, from 2~--~7, 3~--~4, 4~--~1, from 5, 6, 7 and 8 we get 7 atoms, but 9 c.d.\ coincides with the 6 one and 10 with the 1~c.d., from 11 we get 4 atoms, 12~c.d.\ coincides with the 4,~13~-- with the 8 and 14~-- with the~5. We get 7~atoms from 15~c.d., but 16~c.d.\ coincides with the 15 c.d.\ and 17~-- with the~7. We get 1 atom from 18~c.d., from 19~-- 4~atoms. The 20 c.d.\ coincides with the 2~c.d., but 21~c.d.\ corresponds to 4~atoms, 22~-- to 7~atoms, each c.d.~23 and 25 corresponds to 8~atoms. The 24~c.d.\ coincides with the 21~c.d.

As a result we get 94 optimal functions of class $\Omega(M)$ up to layer equivalence. In the same way you can make sure that there exist 182 functions from~$\Omega(M)$ with three critical points up to topological equivalence.

\section{Conclusion}

We've got the following results for the function with isolated critical point on the boundary of the surface which are also isolated critical points of the restriction of this function to the boundary of the surface:
\begin{itemize}\itemsep=0pt
\item local topological presentation about saddle critical point and also about maximum and minimum critical points;
\item the recurrent formula of a number of saddle critical level atoms;
\item criterion of optimality;
\item criterion of topological equivalence of optimal functions in terms of chord diagrams in case of oriented surface with one component of the boundary.
\end{itemize}

Problems of topological equivalence of optimal functions def\/ined on surface with more than one component of the boundary and on non-oriented surface are still open.

\subsection*{Acknowledgements}
This paper partially based on the talks of the f\/irst author given at the AUI's seminars on Topology of functions with isolated critical points on the boundary of a 2-dimensional manifold (March 2--15, 2017, AUI, Vienna, Austria) and partially supported by the project between the Austrian Academy of Sciences and the National Academy of Sciences of Ukraine on Modern Problems in Noncommutative Astroparticle Physics and Categorian Quantum Theory.

\pdfbookmark[1]{References}{ref}
\LastPageEnding

\end{document}